\documentclass[12pt]{amsart}
\usepackage[latin1]{inputenc}
\usepackage{amsbsy}
\usepackage{amscd} 
\usepackage{amsfonts}
\usepackage{amsgen} 
\usepackage{amsmath}
\usepackage{amsopn} 
\usepackage{amssymb}
\usepackage{amstext}
\usepackage{amsthm} 
\usepackage{amsxtra}
\usepackage[all]{xy}
\usepackage[mathscr]{eucal}

\theoremstyle{plain}

\theoremstyle{definition}

\numberwithin{equation}{section}
\newtheorem{definition}{Definition}[section]

\newtheorem{theorem}[definition]{Theorem}

\newtheorem{lemma}[definition]{Lemma}

\renewcommand{\theta}{\vartheta}
\renewcommand{\phi}{\varphi}
\renewcommand{\epsilon}{\varepsilon}

\newcommand{\N}{\mathbb N}

\newcommand{\C}{\mathbb C}
\newcommand{\F}{\mathbb F}

\newcommand{\zw}[2]{\textbf{A.#1 #2.}}

\begin{document}
\title[Decomposition Of Bilinear Forms]{Decomposition of bilinear forms as sums of bounded forms}
\author{Mohamed Elmursi}
\address{Sohag University, Mathematics Institut, Egypt}
\email{m\_elmursi@yahoo.com}
\thanks{\scriptsize{This work has been done in the context of the author's PhD project at the University of Muenster.}}

\begin{abstract}
The problem of decomposition of bilinear forms which satisfy a certain condition has been studied by many authors by example in \cite{H08}: Let $H$ and $K$ be Hilbert spaces and let $A,C \in B(H),B,D\in B(K)$. Assume that $u:H\times K\rightarrow \C$ a bilinear form satisfies 
\[ 
\left|u(x,y)\right|\leq\left\|Ax\right\|\ \left\|By\right\|+\left\|Cx\right\|\left\|Dy\right\| 
\]
for all $ x\in H$ and $y\in K$. Then u can be decomposed as a sum of two bilinear forms 
\[
u=u_1+u_2 
\]
where
\[
\left|u_1(x,y)\right|\leq \left\|Ax\right\|\ \left\|By\right\|,
\left|u_2(x,y)\right|\leq \left\|Cx\right\|\left\|Dy\right\|, \forall x\in H,y\in K.
\] U.Haagerup conjectured that an analogous decomposition as a sum of bounded bilinear forms is not always possible for more than two terms. The aim of current paper is to investigate this problem. In the finite dimensional case, we give a necessary and sufficient criterion for such a decomposition. Finally, we use this criterion to give an example of a sesquilinear form $u$, even on a two-dimensional Hilbert space, which is majorized by the sum of the moduli of three bounded forms $b_1,b_2$ and $b_3$, but can not be decomposed as a sum of three sesquilinear forms $u_i$ where each $u_i$ is majorized by the corresponding $\left|b_i\right|$.

\end{abstract}

\maketitle

\section*{Introduction}

\noindent
A bilinear form on a vector space $V$ is a bilinear mapping $V\times V\rightarrow \F$, where $\F$ is the field of scalars. That is, a bilinear form is a function $B:V\times V\rightarrow \F$ which is a linear in each argument separately. When $\F$ is the field of complex numbers $\C$, one is often more interested in sesquilinear forms, which are similar to bilinear forms but are conjugate linear in one argument. We focus here on the bounded ones. A bilinear form on a normed vector space is bounded if there is a constant $C$ such that for all $u,v\in V$
\[
\left|B(u,v)\right|\leq C\left\|u\right\|\left\|v\right\|.
\]
Let $E$ and $F$ be real or complex vector spaces. In several places in the literature one meets the following situation. One is given a bilinear form $u:E\times F\rightarrow \C$ which can be majorized by the sum of the absolute values of two bounded forms $b_1$ and $b_2$. One then wants to show that $u$ can be decomposed as a sum $u=u_1+u_2$ with $\left|u_1\right|\leq \left|b_1\right|$ and $\left|u_2\right|\leq \left|b_2\right|$. Pisier and Shlyakhtenko \cite{P02} proved such a result for completely bounded forms on exact operator spaces $E\subseteq A$ and $F\subseteq B$ sitting in C*-algebras $A$ and $B$. Let $f_1$,$f_2$ be states on $A$ and $g_1,g_2$ be states on $B$ such that for all $a\in E$ and $b\in F$,
\[
\left|u(a,b)\right|\leq \left\|u\right\|_{ER}(f_1(aa^*)^\frac{1}{2}g_1(b^*b)^\frac{1}{2}+f_2(a^*a)^\frac{1}{2}g_2(bb^*)^\frac{1}{2}).
\]
Then $u$ can be decomposed as
\[
u=u_1+u_2,
\]
where $u_1$ and $u_2$ are bilinear forms satisfying the following inequalities, for all $a\in A$ and $b\in B$:
\begin{equation}
\left|u_1(a,b)\right|\leq\left\|u\right\|_{ER}f_1(aa^*)^\frac{1}{2}g_1(b^*b)^\frac{1}{2}\label{one}
\end{equation}
\begin{equation}
\left|u_2(a,b)\right|\leq \left\|u\right\|_{ER}f_2(a^*a)^\frac{1}{2}g_2(bb^*)^\frac{1}{2}\label{two}.
\end{equation}
In particular,
\[
\left\|u_1\right\|_{cb}\leq \left\|u\right\|_{ER}\text{ and }\left\|u_2^t\right\|_{cb}\leq \left\|u\right\|_{ER},
\]
where $u_2^t(b,a):=u_2(a,b)$, for all $a\in E$ and $b\in F$. Using a similar argument a strengthened version of this result was proved by Xu \cite{X06} for ordinary bilinear forms. Let $H$ and $K$ be Hilbert spaces and let $A,C \in B(H),B,D\in B(K)$. Assume that $u:H\times K\rightarrow \C$ is a bilinear form that satisfies 
\[ 
\left|u(x,y)\right|\leq\left\|Ax\right\|\ \left\|By\right\|+\left\|Cx\right\|\left\|Dy\right\| 
\]
for all $ x\in H$ and $y\in K$. Then u can be decomposed as a sum of two bilinear forms 
\[
u=u_1+u_2 
\]
where
\[
\left|u_1(x,y)\right|\leq \left\|Ax\right\|\ \left\|By\right\|,
\left|u_2(x,y)\right|\leq \left\|Cx\right\|\left\|Dy\right\|, \forall x\in H,y\in K.
\]
The proof there was merely sketched. Later Haagerup-Musat needed the stronger version for bounded bilinear forms on operator spaces \cite{H08}. U.Haa\-ge\-rup conjectured that an analogous decomposition as a sum of bounded bilinear forms is not always possible for more than two terms. It is the aim of the present paper to analyze this problem.

The article is organized as follows. In \textbf{section 1}, we come to the main subject of this paper. We study the problem of decomposition of bounded bilinear forms as in Xu's result. The proof for such a decomposition that we give (following Haagerup) depends on the isomorphism between the projective tensor product $H\hat{\otimes}_\pi K$ where $H$ and $K$ are Hilbert spaces, and the space of trace class operators from the conjugate Hilbert space $\bar{H}$ into $K$ \cite{D93}. First we give a complete proof of Xu's result. Then we consider the case of a bilinear form which is bounded by the sum of three terms. The principal problem considered in this section is the following question: 
Assume that a bilinear form $u$ satisfies the following estimate:
\[
\left|u(x,y)\right|\leq\left\|Ax\right\|\ \left\|By\right\|+\left\|Cx\right\|\left\|Dy\right\|+\left\|Ex\right\|\left\|Fy\right\|
\]
for fixed $A,C,E\in B(H),B,D,F\in B(K)$ and $x\in H,y\in K$. Can $u$ then be decomposed as a sum

\[
u=u_1+u_2+u_3,
\]
such that the forms, $u_1,u_2$ and $u_3$ satisfy the estimates
\[
\left|u_1(x,y)\right|\leq \left\|Ax\right\|\ \left\|By\right\|,\left|u_2(x,y)\right|\leq \left\|Cx\right\|\left\|Dy\right\|, \left|u_3(x,y)\right|\leq \left\|Ex\right\|\left\|Fy\right\|?.
\]
As a first step we find a very special technical condition that allows us to generalize our previous argument to three terms.
\textbf{Section 2} takes a systematic look at the question of decomposing into $n$ bounded terms. Restricting to the finite dimensional case, we give a necessary and sufficient criterion for such a decomposition. The criterion and its proof uses the correspodence between bounded operators on a Hilbert space and sesquilinear forms as well as the trace duality theorem. The proof also uses the Hahn-Banach Theorem applied to convex hulls of certain sets of finite rank operators (the Hahn-Banach theorem had also been used in the proof of Theorem 2.1, the main theorem of this section).
\textbf{Section 3} is the last setion of this paper. Here we use the criterion established in Section 2 to give a counterexample to the
decomposability into three terms. First we give a useful lemma concerning a positive finite rank operator on a finite dimensional Hilbert space. We are then in a position to construct an example of a sesquilinear form $u$, even on a two-dimensional Hilbert space, which is majorized by the sum of the moduli of three bounded forms $b_1,b_2$ and $b_3$, but can not be decomposed as a sum of three sesquilinear forms $u_i$, where each $u_i$ is majorized by the corresponding $\left|b_i\right|$. The counterexample depends on a suitable choice of operators without common eigenvectors.

\section{Decomposition of bilinear forms on Hilbert spaces} \label{SectTwist}

\noindent

In this section, we show how to decompose a bilinear form into two terms by using the isomorphism between the projective tensor product $H\hat{\otimes}_\pi K$ and the space $B_1(\bar{H},K)$ of trace-class operators from $\bar{H}$ into $K$, see for instance \cite{D93}. Next we consider a problem of decomposing a bilinear form to three terms and suggest a solution to the problem by putting a condition on the polar decomposition, related in particular to the ``SVD'' i.e. singular value decomposition, see \cite{W04} and using the same technique as in our main statement.

\subsection{Decomposition of bilinear forms into two terms}
\noindent

In this subsection, we explain how to decompose a bilinear form which satisfies the condition $(*)$ into a sum of two bilinear forms satisfying certain boundedness conditions. We will start with the following theorem which gives the isometric isomorphism between the projective tensor product and the space of trace class operators, see \cite{D93}.

\begin{theorem} 
The map
\[
J:\acute{H}\otimes_\pi H\rightarrow B_1(H)
\]
\[
x'\otimes y\rightarrow x'\underline{\otimes}y
\]
gives an isometric isomorphism
\end{theorem}

\begin{theorem}(\cite{X06})\label{ch2}

Let $H$ and $K$ be Hilbert spaces and let $A,C \in B(H),B,D\in B(K)$. Assume that a bilinear form $u:H\times K\rightarrow \C$ satisfies 
\[ 
\left|u(x,y)\right|\leq\left\|Ax\right\|\ \left\|By\right\|+\left\|Cx\right\|\left\|Dy\right\|(*)
\]
for all $ x\in H$ and $y\in K$. Then u can be decomposed as a sum of two bilinear forms 
\[
 u=u_1+u_2 
\]
 where 
\[
\left|u_1(x,y)\right|\leq \left\|Ax\right\|\left\|By\right\|,
\left|u_2(x,y)\right|\leq \left\|Cx\right\|\left\|Dy\right\|, x\in H,y\in K.
\]
\end{theorem}
\begin{proof}
We consider three cases to prove the claim 
\begin{enumerate}
\item \textbf{Case One}:
We suppose that A=B=I and C,D are invertible operators.

\item \textbf{Case Two}: We suppose that all the operators A,B,C and D are invertible.
We set
\[
v(x,y):=u(A^{-1}x,B^{-1}y).
\]
We can now apply Case One to the bilinear form $v$. We have 
\[
\left|v(x,y)\right|\leq\left\|x\right\|\left\|y\right\|+\left\|CA^{-1}x\right\|\left\|DB^{-1}y\right\|.
\]
Thus by Case One
\[
v=v_1+v_2,
\]
where
\[
\left|v_1(x,y)\right|\leq\left\|x\right\|\left\|y\right\| , \left|v_2(x,y)\right|\leq\left\|CA^{-1}x\right\|\left\|DB^{-1}y\right\|.
\]
Setting
\[
u_i(x,y):=v_i(Ax,By),
\]
Case Two can be reduced to Case One.

\item \textbf{Case Three}: We consider the general case that $A$ is any linear operator in B(H) then $A^*A$ is positive i.e. $A^*A\geq0$. If $\epsilon>0$ is any positive number, then $\epsilon I\geq0$ and $K=A^*A+\epsilon I\geq0$. Clearly $K\geq\epsilon I$. Therefore $K$ is invertible. 
In fact, $sp(K)\subseteq[\epsilon,\infty)$. This means $0\notin sp(K)$ which is equivalent to $K$ being invertible. Thus $A^*A+\epsilon I$ is invertible for $\epsilon>0$. So
\[
A(\epsilon):=(A^*A+ \epsilon I)^\frac{1}{2}
\]
is invertible.
Set
\[
A(\epsilon):=(A^*A+\epsilon1)^\frac{1}{2}, B(\epsilon):=(B^*B+\epsilon1)^\frac{1}{2} 
\]
\[
C(\epsilon):=(C^*C+\epsilon1)^\frac{1}{2} ,D(\epsilon):=(D^*D+\epsilon1)^\frac{1}{2}
\]
Now from the polar decomposition, we can represent any operator $A\in B(H)$ by $A=u\left|A\right|$ where $u$ is a partial isometry on $H$, so
\[
\left\|Ax\right\|=\left\|u\left|A\right|x\right\|\leq\left\|\left|A\right|x\right\|,\text{ }x\in H.
\] 
Since
\[
0\leq\left|A\right|=(A^*A)^\frac{1}{2}\leq(A^*A+\epsilon I)^\frac{1}{2},
\]
it follows that
\[
\left\|\left|A\right| x|\right|\leq\left\|A(\epsilon) x\right\|,\;x\in H.
\]
Hence, we have
\[
\left|u(x,y)\right|\leq\left\|A(\epsilon)x\right\|\ \left\|B(\epsilon)y\right\|+\left\|C(\epsilon)x\right\|\left\|D(\epsilon)y\right\|
\]
Then from Case Two
\begin{equation}\label{limit}
u=u^{\epsilon}_1+u^{\epsilon}_2,
\end{equation}
such that 
\begin{align}\label{estimation}
\left|u_1^\epsilon(x,y)\right|&\leq \left\|A(\epsilon)x\right\|\ \left\|B(\epsilon)y\right\|,\\\label{estim}
\left|u_2^\epsilon(x,y)\right|&\leq \left\|C(\epsilon)x\right\|\left\|D(\epsilon)y\right\|.
\end{align}
Take $0<\epsilon<1$, then $A(\epsilon)\leq A(1)$. Therefore,
\[
\left\|A(\epsilon)\right\|\leq\left\|A(1)\right\|.
\]

In fact, the norms of $\left\|A(\epsilon)\right\|$ are uniformly bounded for all $0<\epsilon<1$. So from the estimations in \eqref{estimation} and \eqref{estim} 
\begin{align}\label{help}
\left\|u_1^{\epsilon}\right\|\leq N_1:=\left\|A(1)\right\|\left\|B(1)\right\|,\\\label{halp}
\left\|u_2^{\epsilon}\right\|\leq N_2:=\left\|C(1)\right\|\left\|D(1)\right\|.
\end{align}
We know from the universal property of projective tensor product (see Proposition 1.4 from \cite{R02}) that , 
\[
Bil(H,K)=(H\otimes_\pi K)^{'},
\]
so, there is $w\in (H\otimes_\pi K)'$ such that 
\[
u(x,y)=w(x\otimes y),\;x\in H \text{ and } y\in K.
\]
Set
\[
M:=N_1+N_2,
\]
and let
\[
S=\left\{w\in(H\otimes_\pi K)':\;\left\|w\right\|\leq M\right\}. 
\]
By Banach-Alaoglu Theorem, $S$ is weak*-compact. Choose two sequences $\left\{w_1^{(n)}\right\}$ and $\left\{w_2^{(n)}\right\}$ in $(H\otimes_\pi K)'$ such that
\[
w_1^{(n)}(x\otimes y)=u_1^{(\frac{1}{n})}(x,y),\;n\in\N
\]
\[
w_2^{(n)}(x\otimes y)=u_2^{(\frac{1}{n})}(x,y),\;n\in\N.
\]
 
So from the definition of $S$ and \eqref{help}, \eqref{halp}, theses sequences $\left\{w_1^{(n)}\right\}$ and $\left\{w_2^{(n)}\right\}$ are in $S$. Hence, they  have convergent subsequences $\left\{w_1^{(n_k)}\right\}$ and $\left\{w_2^{(n_k)}\right\}$ respectively. Thus, when $k\rightarrow \infty$,
\begin{align*}
w_1^{(n_k)}  
&\stackrel{w^*}{\rightharpoonup}w_1,\\
w_2^{(n_k)}&\stackrel{w^*}{\rightharpoonup}w_2.
\end{align*}

Also, $\epsilon\rightarrow0$ when $k\rightarrow\infty$. So from the inequalities in \eqref{estimation} and \eqref{estim} we get:

\begin{align*}
\left|w_1(x\otimes y)\right|&=\left|\lim_{k\rightarrow\infty}w^{(n_k)}_1(x\otimes y)\right|\\
&=\lim_{k\rightarrow \infty}\left|w^{(n_k)}_1(x\otimes y)\right|\leq \lim_{\epsilon\rightarrow0}\left\|A(\epsilon)x\right\|\left\|B(\epsilon)y\right\|\\
&=\left\|\left|A\right|x\right\|\left\|\left|B\right|y\right\|\\
&=\left\|Ax\right\|\left\|By\right\|.
\end{align*}
Therefore, $w_1\in (H\otimes_\pi K)'$.
Similary for $w_2$, we have
\begin{align*}
\left|w_2(x\otimes y)\right|&=\left|\lim_{k\rightarrow\infty}w^{(n_k)}_2(x\otimes y)\right|\\
&=\lim_{k\rightarrow0}\left|w^{(n_k)}_2(x\otimes y)\right|\leq\lim_{\epsilon\rightarrow0}\left\|C(\epsilon)x\right\|\left\|D(\epsilon)y\right\|\\
&=\left\|\left|C\right|x\right\|\left\|\left|D\right|y\right\|\\
&=\left\|Cx\right\|\left\|Dy\right\|.
\end{align*}
Also, $w_2\in (H\otimes_\pi K)'$. Set
\begin{align*}
w_1(x\otimes y)=u_1(x,y),\\
w_2(x\otimes y)=u_2(x,y).
\end{align*}

From \eqref{limit}
\[
u(x,y)=w^{(n_k)}_1(x\otimes y)+w^{(n_k)}_2(x\otimes y).
\]
Now, take the limit point when $k\rightarrow\infty$, we get 
\[
u(x,y)=w_1(x\otimes y)+w_2(x\otimes y).
\]
By construction,
\[
u(x,y)=u_1(x,y)+u_2(x,y)
\]
such that
\[
\left|u_1(x,y)\right|\leq\left\|Ax\right\|\ \left\|By\right\|,
\]
\[
\left|u_2(x,y)\right|\leq\left\|Cx\right\|\left\|Dy\right\|.
\]
and
\[
\left|u(x,y)\right|=\left|u_1(x,y)+u_2(x,y)\right|\leq\left\|Ax\right\|\ \left\|By\right\|+\left\|Cx\right\|\left\|Dy\right\|.
\]
\end{enumerate}

Hence Case Three  follow from Case Two. So
\[
case 1\Rightarrow case 2\Rightarrow case 3
\]
therefore if we prove Case One, we are done. Case One will follow from the next lemma.
\end{proof}

\begin{lemma}(\cite{X06})\label{chapter}

Let $H$ and $K$ be Hilbert spaces and let $C\in L(H)$ and $D\in L(K)$ be invertible. Assume that a bilinear form $ u:H\times K\rightarrow \C$ satisfies
\[
\left|u(x,y)\right|\leq\left\|x\right\|\ \left\|y\right\|+\left\|Cx\right\|\left\|Dy\right\|(*)
\]
for all $x\in H$ and $y \in K $. Then there are bilinear forms $u_1,u_2:H\times K\longrightarrow \C$  such that 
\[
u=u_1+u_2 
\]
and
\[
\left|u_1(x,y)\right|\leq \left\|x\right\|\left\|y\right\|,  \left|u_2(x,y)\right| \leq \left\|Cx\right\| \left\|Dy\right\| 
\]
for all $x\in H$ and $y\in K$
\end{lemma}

\begin{proof}
Let $H \hat{\otimes}_\pi K$ be the projective tensor product of $H$ and $K$. This space is isometrically isomorphic to the space $B_1(\bar{H},K)$ of trace-class operators from $\bar{H}$ into $K$, \cite{D93} . Here, $\bar{H}$ denote the conjugate Hilbert space of $H$. 
Let
\[
w=\sum_{i=1}^m x_i\otimes y_i
\]
be a linear combination of elementary tensors in  $H \hat{\otimes}_\pi K$, then the corresponding linear map $T_w:\bar{H}\rightarrow K$ given by 
\[
T_w\zeta=\sum_{i=1}^m \left\langle x_i|\zeta\right\rangle y_i,\;\zeta\in H
\]
is a finite rank operator and the projective norm $\pi$ of $w$ is given by 
\[
\left\|w\right\|_\pi=\left\|T_w\right\|_1=Tr((T_w^*T_w)^\frac{1}{2}).
\]
Therefore, by Theorem (\cite{C00}, th. 18.13) we can find orthogonal vectors $\left\{\xi_1,...,\xi_n\right\} \in H $ and $\left\{\eta_1,...,\eta_n\right\}\in K$ such that 
\[
w=\sum_{i=1}^n\xi_i\otimes\eta_i
\]
and
\[
\left\|w\right\|_\pi=\sum\left\|\xi_i\right\|^2=\sum\left\|\eta_i\right\|^2,
\]
where $n$ is the rank of $(T_w)$.

In the same way
\[
(C\otimes D)w=\sum^n_{i=1} C\xi_i\otimes D\eta_i
\]
can be written as
\[
(C\otimes D)w=\sum_{i=1}^{n'}\rho_i\otimes\sigma_i.
\]
By invertibility of $C$ and $D$,
\[
n'=rank\left(T_{(C\otimes D)w}\right)=rank(T_w)=n
\]
and
\[
\left\|(C\otimes D)w\right\|_\pi=(\sum_{i=1}^n\left\|\rho_i\right\|^2)=(\sum_{i=1}^n\left\|\sigma_i\right\|^2),
\]

for orthogonal vectors $\left\{\rho_1,...,\rho_n \right\}\in H$ and $\left\{\sigma_1,...,\sigma_n\right\}\in K$.

Since
\[
\sum^n_{i=1} C\xi_i\otimes D\eta_i =\sum_{i=1}^n\rho_i\otimes\sigma_i,
\]
we have by linear independence of each of the sets $\left(C\xi_i\right)_{i=1}^n$, $\left(D\eta_i\right)_{i=1}^n$, $\left(\rho_i\right)_{i=1}^n$ and $\left(\sigma_i\right)_{i=1}^n$ that
\[
C\xi_i=\sum_{j=1}^n\alpha_{ij}\rho_j
\]
and
\[
D\eta_i=\sum^n_{j=1}\beta_{ij}\sigma_j
\]
for unique $\alpha_{ij}$, $\beta_{ij}\in \C$.
Moreover, since
\[
\sum_{i=1}^n\rho_i\otimes\sigma_i= \sum^n_{i=1} C\xi_i\otimes D\eta_i=\sum_{i,j,k=1}^n\alpha_{ij}\beta_{ik}\rho_j\otimes\sigma_k
\]
and from linear independence, we must have
\[
\sum_{j=1}^n\alpha_{ji}\beta_{jk} =\delta_{ik}.
\]
Hence the matrices
\[
\alpha=(\alpha_{ij}) \text{ where }i,j=\left\{1,....,n\right\} \text{ and }\beta=(\beta_{i,j})\text{ where }{i,j}=\left\{1,.......,n\right\}
\]
are invertible and $\beta^{-1}=(\alpha^t)$ where $\alpha^t$ is the transpose of $\alpha$. Write now  
\[
\alpha=UdV
\]
where $U,V\in U(n)$ and $d=diag(d_1,.....,d_n)$ is a diagonal matrix with strictly positive entries $d_1,...,d_n$.
Set
\[
\hat{\xi}_i=\sum_{j=1}^nu^*_{ij}\xi_j=\sum_{j=1}^n\bar{u}_{ji}\xi_j
\]
and
\[
\hat{\rho_i}=\sum^n_{i=1} v_{ij}\rho_j;
\]
then we obtain
\[
C\xi_i=\sum_{j=1}^nu_{ij}d_j\hat{\rho_j}.
\]
Now
\[
\beta=(\alpha^t)^{-1}=(v^tdu^t)^{-1}=\bar{u}d^{-1}\bar{v}.
\]
Then setting
\[
\hat{\eta_i}=\sum_{j=1}^n\bar{u}^*_{ij}\eta_j=\sum_{j=1}^nu_{ji}\eta_j
\]
\[
\hat{\sigma_i}=\sum_{i=1}^n\bar{v}_{ij}\sigma_j,
\]
we obtain similarly
\[
D\eta_i=\sum_{j=1}^n\bar{u}_{ij}d^{-1}_j\hat{\sigma_j}.
\]
Thus
\[
C(\hat{\xi}_i)=\sum_{j=1}^n \bar{u}_{ji} C(\xi_j)=\sum_{j,k=1}^n \bar{u}_{ji}u_{jk} d_k\hat{\rho_k}=\sum_{k=1}^n \delta_{ik}d_k\hat{\rho_k}=d_i\hat{\rho_i}, 
\]
and we obtain similarly,
\[
D(\hat{\eta_i}) =\sum_{j=1}^n u_{ji}D(\eta_j)=\sum_{j,k=1}^n u_{ji} \bar{u}_{jk}d^{-1}_{k} \hat{\sigma_k}=\sum_{k=1}^n \delta_{ik} d^{-1}_k \hat{\sigma_k}=d_{i}^{-1} \hat{\sigma_i}.
\]
Since
\[
\sum_{i=1}^n\hat{\xi}_i\otimes \hat{\eta}_i=\sum_{i,j,k=1}^n \bar{u}_{ji}\xi_j\otimes u_{ki}\eta_k=\sum_{j,k=1}^n \delta_{jk} \xi_j\otimes \eta_k =\sum_{j=1}^n \xi_j\otimes \eta_j= w,
\]
this implies
\[
(C\otimes D)w=\sum_{j=1}^n d_j\hat{\rho}_j\otimes d_j^-1\hat{\sigma}_j=\sum_{j=1}^n\hat{\rho}_j\otimes\hat{\sigma}_j.
\]
Now
\[
\sum_{i=1}^n \left\|\hat{\xi}_j\right\|^2=\sum^n_{i=1}\left\|\sum_{j=1}^n\bar{u}_{ji}\xi_j\right\|=\sum_{i=1}^n\sum_{j=1}^n\left|u_{ij}\right|^2\left\|\xi_i\right\|^2=\sum_{i=1}^n\left\|\xi_i\right\|^2,
\]
and similarly,
\[
\sum_{i=1}^n\left\|\hat{\eta}_j\right\|^2=\sum_{i=1}^n\left\|\sum_{j=1}^nu_{ji}\eta_j\right\|^2=\sum_{i=1}^n\sum_{j=1}^n\left|u_{ji}\right|^2\left\|\eta_j\right\|^2=\sum_{i=1}^n\left\|\eta_i\right\|^2.
\]
Also
\[
\sum_{i=1}^n\left\|\hat{\rho_i}\right\|^2=\sum_{i=1}^n\left\|\sum_{j=1}^n v_{ij}\rho_j\right\|^2=\sum_{i=1}^n\sum_{j=1}^n\left|v_{ij}\right|^2\left\|\rho_j\right\|^2=\sum_{i=1}^n\left\|\rho_i\right\|^2
\]
and similarly,
\[
\sum_{i=1}^n\left\|\hat{\sigma_i}\right\|^2=\sum_{i=1}^n\left\|\sum_{j=1}^n \bar{v}_{ij}\sigma_j\right\|^2=\sum_{i=1}^n\sum_{j=1}^n\left|\bar{v}_{ij}\right|^2\left\|\sigma_j\right\|^2=\sum_{i=1}^n\left\|\sigma_i\right\|^2.
\]
Therefore
\[
\left\| w  \right\|_\pi = \sum_{j=1}^n \left\| \hat{\xi}_j \right\| ^2=\sum_{j=1}^n\left\|\hat{\eta}_j\right\|^2
\]
\[
\left\|(C\otimes D)w\right\|_\pi=\sum^n_{j=1}\left\|\hat{\rho}_j\right\|^2=\sum_{j=1}^n\left\|\hat{\sigma}_j\right\|^2.
\]
Hence
\begin{align*}
\left|\ \sum_{i=1}^n u(x_i,y_i) \right|\ &= \left|\ \sum_{i=1}^n u(\hat{\xi_i},\hat{\eta_i}) \right|\\\
&\leq \sum_{i=1}^n\left\|\hat{\xi_i}\right\| \left\| \hat{\eta_i}\right\|+\sum _{i=1}^n\left\|C\hat{\xi_i}\right\|\left\|D\hat{\eta_i}\right\|\\
&=\sum_{i=1}^n\left\|\hat{\xi_i}\right\| \left\| \hat{\eta_i}\right\| +\sum_{i=1}^n\left\|\hat{\rho_i}\right\| \left\| \hat{\sigma_i}\right\|\\
&\leq\sum_{i=1}^n(\left\|\hat{\xi_i}\right\|)^2)^\frac{1}{2} (\sum_{i=1}^n\left\|\eta_i\right\|^2)^\frac{1}{2} +\sum_{i=1}^n(\left\|\hat{\rho_i}\right\|)^2)^\frac{1}{2}
(\sum_{i=1}^n\left\|\sigma_i\right\|^2)^\frac{1}{2}
\end{align*}
\begin{equation}\label{trick}
= \left\|\sum x_i\otimes y_i\right\|_\pi + \left\| \sum Cx_i\otimes Dy_i\right\|_\pi 
\end{equation}
If $V$ and $W$ are Banach spaces we denote by $V\oplus_1 W$ the direct sum of $V$ and $W$ endowed with the norm 

\[
\left\|(v,w)\right\|=\left\|v\right\|+\left\|w\right\|.
\]
Let $E$ be the linear span of all vectors $(x\otimes y,C(x)\otimes D(y))$ in $(H \hat{\otimes}_\pi K)\oplus_1 (H \hat{\otimes}_\pi K)$ where $x\in H$ and $y\in K$. According to the above estimate in \eqref{trick} we find a bounded linear functional $w \in E^* $ with $\left\|w\right\| \leq 1$ such that 
\[
u(x,y)= w((x\otimes y,C(x)\otimes D(y))
\]
for all $x\in H$ and $y \in K$. By the Hahn-Banach Theorem there exists a bonunded linear functional $\tilde{w}$ on $(H \hat{\otimes}_\pi K)\oplus_1 (H \hat{\otimes}_\pi K)$   with  $\left\|\tilde{w}\right\|=\left\|w\right\| \leq 1$ extending $w$. 
We set
\[
u_1(x,y)=\tilde{w}((x\otimes y,0)), u_2(x,y)=\tilde{w}(0,C(x)\otimes D(y)).
\] 
By construction we have 
\[
u=u_1+u_2.
\]
Moreover
\[
\left|u_1(x,y)\right| \leq \left\|\tilde{w}\right\|\left\|x\right\|\left\|y\right\|\leq \left\|x\right\|\left\|y\right\|
\]
and
\[
\left|u_2(x,y)\right|\leq \left\|\tilde{w}\right\|\left\|C(x)\right\|\left\|D(y)\right\|\leq \left\|C(x)\right\|\left\|D(y)\right\|.
\]
This yields the claim.
\end{proof}

\subsection{Decomposition of bilinear form into three terms} 
\noindent

In this subsection, we consider the case of six bounded operators by adding $E,F$ to $A,B$ and $C,D$ and ask: if we can majorize the bilinear form $u$ with bounded forms,
\[
\left|u(x,y)\right|\leq\left\|Ax\right\|\ \left\|By\right\|+\left\|Cx\right\|\left\|Dy\right\|+\left\|Ex\right\|\left\|Fy\right\|,
\]
then can we decompose $u$ as a sum of three bilinear forms
\[
u=u_1+u_2+u_3
\]
which satisfy the boundedness conditions
\[
\left|u_1(x,y)\right|\leq \left\|Ax\right\|\ \left\|By\right\|,\left|u_2(x,y)\right|\leq \left\|Cx\right\|\left\|Dy\right\|, \left|u_3(x,y)\right|\leq \left\|Ex\right\|\left\|Fy\right\|?
\]
The answer is No in general; see Section 3. But in this section, we look for a particular condition which make it possible to solve this problem.

Singular value decomposition plays an important role in the proof of the main statement as in Lemma 2.2. As we have seen in Section 1, by the singular value decomposition for the complex matrix $\alpha$:
\[
\alpha=UdV
\]
we can get new sets of vectors, $\left\{\hat{\xi_1},...,\hat{\xi_n}\right\}$, $\left\{\hat{\rho_1},...,\hat{\rho_n}\right\}$ and $\left\{\hat{\eta_1},...,\hat{\eta_n}\right\}$, $\left\{\hat{\sigma_1},...,\hat{\sigma_n}\right\}$ in H and K respectively. By using those sets of vectors we can diagonalize the operators $C$ and $D$,
\[
C(\hat{\xi}_i)=d_i\hat{\rho_i}, 
\]
\[
D(\hat{\eta_i})=d_{i}^{-1} \hat{\sigma_i},
\]

and this helps us to get this estimate in Theorem \ref{trick}. By the Hahn-Banach Theorem, we can obtain from this a decomposition of a bilinear form into two terms.

We can use the same technique as in the proof of Lemma 2.2 to apply the operator $E\otimes F$ on the same vector $w$ in $H\hat{\otimes_\pi}K$ and calculate the SVD for the new complex matrix 
\[
\tilde{\alpha}=SeT
\]
where $e$ is a diagonal matrix and $S,T$ are unitary matrices to get other sets of vectors,
$\left\{\tilde{\xi_1},\ldots,\tilde{\xi_n}\right\},\left\{\check{\tilde{\rho_1}},\ldots,\check{\tilde{\rho_n}}\right\}$ in H and $\left\{\check{\eta_1},\ldots,\check{\eta_n}\right\},\left\{\check{\tilde{\sigma_1}},\ldots,\check{\tilde{\sigma_n}}\right\}$ in K. Therefore we can diagonalize the operators $E$ and $F$ as well,
\[
E(\check{\xi_i})=e_i\check{\tilde{\rho_i}},
\]
\[
F(\check{\eta_i})=e_{i}^{-1}\check{\tilde{\sigma_i}}
\]

In order to make the diagonalization of $C,D$ and $E,F$ compatible, we must put a very special technical condition on the unitary matrix in the SVD i.e.
\begin{equation}\label{condition}
U=S
\end{equation}
If this condition happens, we can decompose the bilinear form into three terms as we will see in the following argument. Actually, we can consider also three cases as in Theorem \ref{ch2} as follows.
\begin{enumerate}
\item \textbf{Case One}:
We suppose that A=B=I and C,D,E and F are invertible operators.
\item \textbf{Case Two}:
We suppose that all the operators A,B,C,D and E,F are invertible.
Then set
\[
v(x,y):=u(A^{-1}x,B^{-1}y).
\] 
By applying Case One to the bilinear form $v$, we have
\[
\left|v(x,y)\right|\leq\left\|x\right\|\left\|y\right\|+\left\|CA^{-1}x\right\|\left\|DB^{-1}y\right\|+\left\|EA^{-1}x\right\|\left\|FB^{-1}y\right\|
\]
Then by the same case,
\[
v=v_1+v_2+v_3,
\]
where
\begin{align*}
\left|v_1(x,y)\right|&\leq\left\|x\right\|\left\|y\right\|,\\
\left|v_2(x,y)\right|&\leq\left\|CA^{-1}x\right\|\left\|DB^{-1}y\right\|\text{ and }\\ \left|v_3(x,y)\right|&\leq\left\|EA^{-1}x\right\|\left\|FB^{-1}y\right\|. 
\end{align*}
Setting
\[
u_i(x,y):=v_i(Ax,By),
\]
Case Two can be reduced to Case One. 

\item \textbf{Case Three}: We consider the general case that $A$\ is any linear operator in B(H). Then $A^*A$ is positive i.e. $A^*A\geq0$. So $K=A^*A+\epsilon I\geq0$ for any $\epsilon>0$, hence $K\geq\epsilon I$. Therefore $K$ is invertible. In fact, $sp(K)\subseteq[\epsilon,\infty)$. This means $0\notin sp(K)$ which is equivalent to $K$ being invertible.
Thus
\[
A(\epsilon):=(A^*A+ \epsilon I)^\frac{1}{2}
\]
is invertible for $\epsilon>0$.
Set
\[
A(\epsilon):=(A^*A+\epsilon1)^\frac{1}{2} , B(\epsilon):=(B^*B+\epsilon1)^\frac{1}{2} 
\]
\[
C(\epsilon):=(C^*C+\epsilon1)^\frac{1}{2} ,D(\epsilon):=(D^*D+\epsilon1)^\frac{1}{2}
\]
\[
E(\epsilon):=(E^*E+\epsilon1)^\frac{1}{2} ,F(\epsilon):=(F^*F+\epsilon1)^\frac{1}{2}.
\]
Now from the polar decomposition, we can represent any operator $A\in B(H)$ by $A=u\left|A\right|$ where $u$ is a partial isometry on $H$, so
\[
\left\|Ax\right\|=\left\|u\left|A\right|x\right\|\leq\left\|\left|A\right|x\right\|,\text{ }x\in H.
\]
Since
\[
0\leq\left|A\right|=(A^*A)^\frac{1}{2}\leq(A^*A+\epsilon I)^\frac{1}{2},
\]
so
\[
\left\|\left|A\right| x|\right|\leq\left\|A(\epsilon) x\right\|,\;x\in H.
\]
Therefore, we have
\[
\left|u(x,y)\right|\leq\left\|A(\epsilon)x\right\|\ \left\|B(\epsilon)y\right\|+\left\|C(\epsilon)x\right\|\left\|D(\epsilon)y\right\|+\left\|E(\epsilon)x\right\|\left\|F(\epsilon)y\right\|.
\]
Then from Case Two

\begin{align}\label{three}
u=u_1^\epsilon+u_2^\epsilon+u_3^\epsilon,
\end{align}
such that
\begin{align}\label{es1}
\left|u_1^\epsilon(x,y)\right|&\leq \left\|A(\epsilon)x\right\|\ \left\|B(\epsilon)y\right\|,\\\label{es2}
\left|u_2^\epsilon(x,y)\right|&\leq \left\|C(\epsilon)x\right\|\left\|D(\epsilon)y\right\|,\\\label{es3}
\left|u_3^\epsilon(x,y)\right|&\leq \left\|E(\epsilon)x\right\|\left\|F(\epsilon)y\right\|.
\end{align}

By the same assumption for $\epsilon$ in Theorem \ref{ch2}, we can get from the estimations in \eqref{es1}, \eqref{es2} and \eqref{es3} that,
\begin{align}\label{inq1}
\left\|u^\epsilon_1\right\|\leq N_1&:=\left\|A(1)\right\|\left\|B(1)\right\|\\\label{inq2}
\left\|u_2^\epsilon\right\|\leq N_2&:=\left\|C(1)\right\|\left\|D(1)\right\|\\\label{inq3}
\left\|u_3^\epsilon\right\|\leq N_3&:=\left\|E(1)\right\|\left\|F(1)\right\|.
\end{align}

Set, now 
\[
M:=N_1+N_2+N_3
\]
and let
\[
S=\left\{w\in(H\otimes_\pi K)':\;\left\|w\right\|\leq M\right\}. 
\]
Also, by Banach-Alaoglu Theorem, $S$ is weak*-compact. Choose three sequences $\left\{w^{(n)}_1\right\}$, $\left\{w^{(n)}_2\right\}$ and $\left\{w^{(n)}_3\right\}$ in $(H\otimes_\pi K)'$, such that 

\begin{align*}
w_1^{(n)}(x\otimes y)&=u_1^{(\frac{1}{n})}(x,y),\;n\in\N\\
w_2^{(n)}(x\otimes y)&=u_2^{(\frac{1}{n})}(x,y),\;n\in\N,\\
w_3^{(n)}(x\otimes y)&=u_3^{(\frac{1}{n})}(x,y),\;n\in\N.
\end{align*}

So from the definition of $S$ and \eqref{inq1}, \eqref{inq2} and \eqref{inq3}. They are in $S$. Thus, they have convergent subsequences $w_1^{(n_k)}$, $w_2^{(n_k)}$ and $w_3^{(n_k)}$ respectively. When $k\rightarrow \infty$,
\begin{align*}
w_1^{(n_k)}&\stackrel{w^*}{\rightharpoonup}w_1,\\
w_1^{(n_k)}&\stackrel{w^*}{\rightharpoonup} w_2,\\
w_3^{(n_k)}&\stackrel{w^*}{\rightharpoonup} w_3,
\end{align*}

Also, $\epsilon\rightarrow0$ when $k\rightarrow0$. So from the inequalities in \eqref{es1}, \eqref{es2} and \eqref{es3}, we have

\begin{align*}
\left|w_1(x\otimes y)\right|&=\left|\lim_{k\rightarrow\infty}w^{(n_k)}_1(x\otimes y)\right|\\
&=\lim_{k\rightarrow \infty}\left|w^{(n_k)}_1(x\otimes y)\right|\leq \lim_{\epsilon\rightarrow0}\left\|A(\epsilon)x\right\|\left\|B(\epsilon)y\right\|\\
&=\left\|\left|A\right|x\right\|\left\|\left|B\right|y\right\|\\
&=\left\|Ax\right\|\left\|By\right\|.
\end{align*}
Therefore, $w_1\in(H\otimes_\pi K)'$. Similarly for $w_2$ and $w_3$, we have 

\begin{align*}
\left|w_2(x\otimes y)\right|&=\left|\lim_{k\rightarrow\infty}w^{(n_k)}_2(x\otimes y)\right|\\
&=\lim_{k\rightarrow0}\left|w^{(n_k)}_2(x\otimes y)\right|\leq\lim_{\epsilon\rightarrow0}\left\|C(\epsilon)x\right\|\left\|D(\epsilon)y\right\|\\
&=\left\|\left|C\right|x\right\|\left\|\left|D\right|y\right\|\\
&=\left\|Cx\right\|\left\|Dy\right\|,
\end{align*}
and
\begin{align*}
\left|w_3(x\otimes y)\right|&=\left|\lim_{k\rightarrow\infty}w^{(n_k)}_3(x\otimes y)\right|\\
&=\lim_{k\rightarrow0}\left|w^{(n_k)}_3(x\otimes y)\right|\leq\lim_{\epsilon\rightarrow0}\left\|E(\epsilon)x\right\|\left\|F(\epsilon)y\right\|\\
&=\left\|\left|E\right|x\right\|\left\|\left|F\right|y\right\|\\
&=\left\|Ex\right\|\left\|Fy\right\|.
\end{align*}
Therefore, also $w_2,w_3\in(H\otimes_\pi K)'$. Hence there are $u_1,u_2$ and $u_3\in Bil(H,K)$ such that
\begin{align*}
w_1(x\otimes y)&=u_1(x,y),\\
w_2(x\otimes y)&=u_2(x,y),\\
w_3(x\otimes y)&=u_3(x,y).
\end{align*}

From \eqref{three}

\[
u(x,y)=w^{(n_k)}_1(x\otimes y)+w^{(n_k)}_2(x\otimes y)+w^{(n_k)}_3(x\otimes y),
\]
take the limit point for \eqref{three}, we have
\[
u(x,y)=w_1(x\otimes y)+w_2(x\otimes y)+w_3(x \otimes y).
\]
By construction
\[
u(x,y)=u_1(x,y)+u_2(x,y)+u_3(x,y),
\]
such that
\begin{align*}
\left|u_1(x,y)\right|&\leq \left\|Ax\right\|\left\|By\right\|,\\
\left|u_2(x,y)\right|&\leq \left\|Cx\right\|\left\|Dy\right\|,\\
\left|u_1(x,y)\right|&\leq \left\|Ax\right\|\left\|By\right\|.
\end{align*}
and
\[
\left|u(x,y)\right|\leq\left\|Ax\right\|\ \left\|By\right\|+\left\|Cx\right\|\left\|Dy\right\|+ \left\|Ex\right\|\left\|Fy\right\|.
\]
Thus Case Three follows from Case Two. So
\[
case 1\Rightarrow case 2\Rightarrow case 3
\]
therefore we can reduce the problem to Case One.
\end{enumerate}

From Lemma \ref{chapter}, we can find orthogonal vectors $\left\{\xi_1,...,\xi_n\right\}\in H$ and $\left\{\eta_1,...,\eta_n\right\}\in K$ such that 
\[
w=\sum_{i=1}^n\xi_i\otimes\eta_i
\]
and
\[
\left\|w\right\|_\pi=(\sum_{i=1}^n\left\|\xi_i\right\|^2)=(\sum_{i=1}^n\left\|\eta_i\right\|^2),
\]
where $n$ is the rank of $(T_w)$. Since $C$ and $D$ are invertible, we obtain similarly 
\[
(C\otimes D)w=\sum_{i=1}^n\rho_i\otimes\sigma_i,
\]
for orthogonal vectors $\left\{\rho_1,...,\rho_n\right\}\in H$ and $\left\{\sigma_1,...,\sigma_n\right\}\in K$ such that 
\[
\left\|(C\otimes D)w\right\|_\pi=(\sum_{i=1}^n\left\|\rho_i\right\|^2)=(\sum_{i=1}^n\left\|\sigma_i\right\|^2).
\]
Since
\[
\sum^n_{i=1} C\xi_i\otimes D\eta_i =\sum_{i=1}^n\rho_i\otimes\sigma_i,
\]
there are complex numbers $\alpha_{ij}$ and $\beta_{ij} $ such that 
\[
C\xi_i=\sum_{j=1}^n\alpha_{ij}\rho_j , D\eta_i=\sum^n_{j=1}\beta_{ij}\sigma_j.
\]
We have
\[
\sum_{j=1}^n\rho_j\otimes\sigma_j= \sum^n_{i=1} C\xi_i\otimes D\eta_i=\sum_{i,j,k=1}^n\alpha_{ij}\beta_{ik}\rho_j\otimes\sigma_k
\]
and from linear independence, we conclude
\[
\sum_{j=1}^n\alpha_{ji}\beta_{jk} =\delta_{ik}.
\]
Hence the matrices
\[
\alpha=(\alpha_{ij}), i,j=1,....,n\text{ and }\beta=(\beta_{i,j}),{i,j}=1,.......,n 
\]
are invertible and $\beta^{-1}=(\alpha^t)$ where $\alpha^t$ is the transpose of $\alpha$. Now write
\[
\alpha=UdV
\]
where $U,V\in U(n)$ and $d=diag(d_1,.....,d_n)$ is a diagonal matrix with strictly positive entries $d_1,...,d_n$.
Set
\[
\hat{\xi}_i=\sum_ju^*_{ij}\xi_j=\sum_j\bar{u}_{ji}\xi_j
\]
\[
\hat{\eta_i}=\sum_j\bar{u}^*_{ij}\eta_j=\sum_ju_{ji}\eta_j
\]
and
\[
\hat{\rho_i}=\sum v_{ij}\rho_j ,\hat{\sigma_i}=\sum \bar{v}_{ij}\sigma_j.
\]
Then we obtain
\[
C\xi_i=\sum_{j=1}^nu_{ij}d_j\hat{\rho_j}
\]
and hence
\[
C(\hat{\xi}_i)=\sum_{j=1}^n \bar{u}_{ji} C(\xi_j)=\sum_{j,k=1}^n \bar{u}_{ji}u_{jk} d_k\hat{\rho_k}=\sum_{k=1}^n \delta_{ik}d_k\hat{\rho_k}=d_i\hat{\rho_i}.
\]
Since
\[
\beta=(\alpha^t)^{-1}=(v^tdu^t)^{-1}=\bar{u}d^{-1}\bar{v},
\]
we obtain similarly
\[
D(\hat{\eta_i}) =\sum_{j=1}^n u_{ji}D(\eta_j)=\sum_{j,k=1}^n u_{ji} \bar{u}_{jk}d^{-1}_{k} \hat{\sigma_k}=\sum_{k=1}^n \delta_{ik} d^{-1}_k \hat{\sigma_k}=d_{i}^{-1} \hat{\sigma_i}.
\]
Since
\[
\sum_{i=1}^n\hat{\xi}_i\otimes \hat{\eta}_i=\sum_{i,j,k=1}^n \bar{u}_{ji}\xi_j\otimes u_{ki}\eta_k=\sum_{j,k=1}^n \delta_{jk} \xi_j\otimes \eta_k =\sum_{j=1}^n \xi_j\otimes \eta_j= w,
\]
this implies
\[
(C\otimes D)w=\sum_{j=1}^n d_j\hat{\rho}_j\otimes d_j^{-1}\hat{\sigma}_j=\sum_{j=1}^n\hat{\rho}_j\otimes\hat{\sigma}_j,
\]
and we have
\[
\left\|w\right\|_\pi = \sum_{j=1}^n \left\| \hat{\xi}_j \right\| ^2=\sum_{j=1}^n\left\|\hat{\eta}_j\right\|^2
\]
\[
\left\|(C\otimes D)w\right\|_\pi=\sum^n_{j=1}\left\|\hat{\rho}_j\right\|^2=\sum_{j=1}^n\left\|\hat{\sigma}_j\right\|^2.
\]
As the same way we can apply the operator $E\otimes F$ to $w$ to obtain
\[
(E\otimes F)w=\sum_{i=1}^n\tilde{\rho_i}\otimes\tilde{\sigma}_i=  \sum^n_{i=1} E\xi_i\otimes F\eta_i
\]
so
\[
E\xi_{i}=\sum^n_{j=1} \tilde{\alpha_{ij}}\tilde{\rho_j} , F\eta_i=\sum_{j}\tilde{\beta_{ij}}\tilde{\sigma_j}.
\]
Now write
\[
\tilde{\alpha}=SeT
\]
where $S,T$ are unitaries matrices and $e=diag(e_1,....,e_n)$ is a diagonal matrix with strictly positive entries. 
Set
\[
\check{\xi}_i=\sum_j\bar{s}_{ji}\xi_j,
\]
\[
\check{\eta_i}=\sum_js_{ji}\eta_j,
\]
and
\[
\check{\tilde{\rho_i}}=\sum t_{ij}\tilde{\rho_j},\;\check{\tilde{\sigma_i}}=\sum \bar{t}_{ij}\tilde{\sigma_j}.
\]
Then we obtain
\[
E\xi_i=\sum_{j=1}^ns_{ij}e_j\check{\tilde{\rho_j}},\;F\eta_i=\sum_j\bar{s_{ij}}e^{-1}_j\check{\tilde{\sigma_j}}
\]
\[
E(\check{\xi_i})=\sum_{j=1}^n \bar{s}_{ji} E(\xi_j)=\sum_{j,k=1}^n \bar{s}_{ji}s_{jk} e_k\check{\tilde{\rho_k}}=\sum_{k=1}^n \delta_{ik}e_k\check{\tilde{\rho_k}}=e_i\check{\tilde{\rho_i}}.
\]
Since
\[
\bar{\beta}=(\bar{\alpha}^t)^{-1}=\bar{s}e^{-1}\bar{t}
\]
\[
F(\check{\eta_i}) =\sum_{j=1}^n s_{ji}F(\eta_j)=\sum_{j,k=1}^n s_{ji} \bar{s_{jk}}e^{-1}_{k} \check{\tilde{\sigma_k}}=\sum_{k=1}^n \delta_{ik} e^{-1}_k\check{\tilde{\sigma_k}}=e_{i}^{-1} \check{\tilde{\sigma_i}}
\]
and this implies
\[
\sum\check{\xi}_i\otimes \check{\eta}_i=\sum_{i,j,k=1}^n \bar{s}_{ji}\xi_j\otimes s_{ki}\eta_k=\sum_{j,k=1}^n\delta_{jk}\xi_j\otimes\eta_k =\sum_{j=1}^n \xi_j\otimes \eta_j= w
\]
\[
(E\otimes F)w=\sum_{j=1}^n e_j\check{\tilde{\rho}}_j\otimes e_j^{-1}\check{\tilde{\sigma}}_j=\sum_{j=1}^n\check{
\tilde{\rho}}_j\otimes\check{\tilde{\sigma}}_j.
\]
Then we have 
\[
\left\|w\right\|_\pi = \sum_{j=1}^n \left\|\check{\xi_j} \right\| ^2=\sum_{j=1}^n\left\|\check{\eta_j}\right\|^2
\]
\[
\left\|(E\otimes F)w\right\|_\pi=\sum^n_{j=1}\left\|\check{\tilde{\rho_j}}\right\|^2=\sum_{j=1}^n\left\|\check{\tilde{\sigma_j}}\right\|^2.
\]
Now by using our condition in \eqref{condition}, therefore
\[
C(\check{\xi_i})=\sum_{j=1}^n \bar{s}_{ji} C(\xi_j)=\sum_{j,k=1}^n \bar{s}_{ji}u_{jk} d_k\hat{\rho_k}=\sum_{k=1}^n \delta_{ik}d_k\hat{\rho_k}=d_i\hat{\rho_i}, 
\]
and similarly 
\[
D(\check{\eta_i}) =\sum_{j=1}^n s_{ji}D(\eta_j)=\sum_{j,k=1}^n s_{ji} \bar{u}_{jk}d^{-1}_{k} \hat{\sigma_k}=\sum_{k=1}^n \delta_{ik} d^{-1}_k \hat{\sigma_k}=d_{i}^{-1} \hat{\sigma_i}.
\]
Hence
\begin{align*}
\left|\sum_{i=1}^n u(x_i,y_i)\right|&=\left|\sum_{i=1}^n u(\check{\xi_i},\check{\eta_i})\right|\\
&\leq\sum_{i=1}^n\left\|\check{\xi_i}\right\| \left\|\check{\eta_i}\right\|+\sum_{i=1}^n\left\|C\check{\xi_i}\right\|\left\|D\check{\eta_i}\right\|+\sum _{i=1}^n\left\|E\check{\xi_i}\right\|\left\|F\check{\eta_i}\right\|\\
&=\sum_{i=1}^n\left\|\check{\xi_i}\right\|\left\|\check{\eta_i}\right\| +\sum_{i=1}^n\left\|\hat{\rho_i}\right\|\left\| \hat{\sigma_i}\right\|+\sum_{i=1}^n\left\|\check{\tilde{\rho_i}}\right\|\left\|\check{\tilde{\sigma_i}}\right\|\\
&\leq\sum_{i=1}^n(\left\|\hat{\xi_i}\right\|)^2)^\frac{1}{2} (\sum_{i=1}^n\left\|\eta_i\right\|^2)^\frac{1}{2} +\sum_{i=1}^n(\left\|\hat{\rho_i}\right\|)^2)^\frac{1}{2}(\sum_{i=1}^n\left\|\hat{\sigma_i}\right\|^2)^\frac{1}{2}+\\
&+\sum_{i=1}^n(\left\|\hat{\bar\rho_i}\right\|)^2)^\frac{1}{2}(\sum_{i=1}^n\left\|\hat{\bar{\sigma_i}}\right\|^2)^\frac{1}{2}
\end{align*}
\begin{equation}\label{estimate}
\hspace{5.3em}=\left\|\sum x_i\otimes y_i\right\|_\pi + \left\| \sum Cx_i\otimes Dy_i\right\|_\pi+\left\| \sum Ex_i\otimes Fy_i\right\|_\pi.
\end{equation}
If $V$, $W$ and $Z$ are Banach spaces we denote by $V\oplus_1 W\oplus_1 Z$ the direct sum of $V$, $W$ and $Z$ endowed with the norm 
\[
\left\|(v,w,z)\right\|=\left\|v\right\|+\left\|w\right\|+\left\|z\right\|.
\]
Let $E$ be the linear span of all vectors $(x\otimes y,C(x)\otimes D(y),E(x)\otimes F(y))$ in $(H \hat{\otimes}_\pi K)\oplus_1 (H \hat{\otimes}_\pi K)\oplus_1 (H \hat{\otimes}_\pi K)$ where $x\in H$ and $y\in K$. According to the above estimate in \eqref{estimate} we find a bounded linear functional $w \in E^* $ with $\left\|w\right\| \leq 1$ such that 
\[
u(x,y)= w((x\otimes y,C(x)\otimes D(y),E(x)\otimes F(y)),
\]
for all $x\in H$ and $y \in K$. By the Hahn-Banach Theorem there exists a bonunded linear functional $\tilde{w}$ on $(H \hat{\otimes}_\pi K)\oplus_1 (H \hat{\otimes}_\pi K)\oplus_1(H \hat{\otimes}_\pi K)$ with $\left\|\tilde{w}\right\|=\left\|w\right\| \leq 1$ extending $w$.
We set 
\begin{align*}
u_1(x,y)&=\tilde{w}((x\otimes y,0,0)),\\ 
u_2(x,y)&=\tilde{w}(0,C(x)\otimes D(y),0),\\
u_3(x,y)&=\tilde{w}(0,0,E(x)\otimes F(y)),
\end{align*}
by construction we have 
\[
u=u_1+u_2+u_3.
\]
Moreover
\[
\left|u_1(x,y)\right| \leq \left\|\tilde{w}\right\|\left\|x\right\|\left\|y\right\|\leq \left\|x\right\|\left\|y\right\|,
\]
\[
\left|u_2(x,y)\right|\leq \left\|\tilde{w}\right\|\left\|C(x)\right\|\left\|D(y)\right\|\leq \left\|C(x)\right\|\left\|D(y)\right\|,
\]
\[
\left|u_3(x,y)\right|\leq \left\|\tilde{w}\right\|\left\|E(x)\right\|\left\|F(y)\right\|\leq \left\|E(x)\right\|\left\|F(y)\right\|.
\]

\section{Decomposition of bilinear forms into n terms}\label{SectIdeale}

\noindent


In this section, we discuss a problem of decomposing into n bounded terms. Turning to the finite dimensional case, we find a criterion to make the decomposition possible for n terms. We will work with sesquilinear forms instead of bilinear ones. We begin with a lemma which provides a bijective correspondence between bounded operators on $H$ and bounded sesquilinear forms. This is well known, see \cite{P89}.
\begin{lemma}\cite{P89}\label{bijective}
There is a bijective correspondence $A\longmapsto b_A$ between bounded operators on $H$ and bounded sesquilinear forms given by 
\[
b_A(x,y)=\left\langle Ax|y\right\rangle x,y \in H.
\]
One has 
\[
\left\|A\right\|=sup\left\{\left|b_A(x|y)\right||
\left\|x\right\|,\left\|y\right\|\leq 1\right\}
\]
\end{lemma}
Now we come to the main theorem in this section.
\begin{theorem}\label{ch5}
Let $H$ be a finite-dimensional Hilbert space and let $A_2,\dots,A_n$ and $B_2,\ldots, B_n$ be invertible operators in $B(H)$. Assume that $U\in B(H)$ is a bounded operator which satisfies
\[
\left|\left\langle Ux|y\right\rangle\right|\leq\left\|x\right\|\ \left\|y\right\|+\left\|A_2x\right\|\left\|B_2y\right\|+\ldots+
\left\|A_nx\right\|\left\|B_ny\right\|
\]
for all $x,y\in H$.
 
Then the following two conditions are equivalent:
\begin{itemize}
\item[(a)]
$U$ can be split into a sum of n-terms 
\[
U=U_1+U_2+\ldots+U_n,\;U_i\in B(H),
\]
such that
\begin{align*}
\left|\left\langle U_1x|y\right\rangle\right| &\leq \left\|x\right\|\left\|y\right\|,\\
\left|\left\langle U_2x|y\right\rangle\right| &\leq \left\|A_2x\right\|\left\|B_2y\right\|,\\
&\hspace{.5em}\vdots\\
\left|\left\langle U_nx|y\right\rangle\right| &\leq \left\|A_nx\right\|\left\|B_ny\right\|
\end{align*}
for all $ x,y\in H$.

\item[(b)]
If we set 
\[
K=\left\{x\underline{\otimes}y:\:\left\|x\right\|\left\|y\right\|+\left\|A_2x\right\|\left\|B_2y\right\|+\cdots+\left\|A_nx\right\|\left\|B_ny\right\|\leq 1\right\}
\]
(where $x\underline{\otimes}y$ denotes a rank one operator, see the appendix)
and  
\[
\Delta=\left\{T\in B(H):\;\left\|T\right\|_1+\left\|A_2TB_2^*\right\|_1+\cdots+\left\|A_nTB_n^*\right\|_1\leq 1\right\},
\]
(where $\left\|S\right\|_1=tr\left|S\right|$ denotes the trace class norm of $S$), then
\[
conv(K)=\Delta
\]
(conv denotes the convex hull).
\end{itemize}
\end{theorem}
\begin{proof}
$(b)\Rightarrow(a)$: For any bounded operator $U\in B(H)$, by trace duality, see \cite{D95}. We can associate a linear functional $\phi$ on $B(H)$ such that, 
\[
\phi(T)=tr(UT),\text{ }T\in B(H).
\]
Hence,
\[
\phi(x\underline{\otimes}y)=tr(Ux\underline{\otimes}y)=\left\langle Ux|y\right\rangle,
\]
and therefore,
\[
\left|\left\langle Ux|y\right\rangle\right|=\left|\phi(x\underline{\otimes}y)\right|\leq\left\|x\underline{\otimes} y\right\|_1+\left\|A_2x \underline{\otimes} B_2y\right\|_1+....+ \left\|A_nx \underline{\otimes}B_ny\right\|_1\leq 1
\]
for all $x\underline{\otimes}y\in K$. By assumption,
\[
conv(K)=\Delta.
\]
So any $T\in \Delta$ has the form $T=\sum_{i=1}^n \lambda_ix_i\underline{\otimes} y_i$ where,
\[
\sum_{i=1}^n\lambda_i=1 \text{ and } x_i \underline{\otimes}y_i\in K,
\]

Therefore,
\[
\left|\sum_{i=1}^n\left\langle Ux_i|y_i\right\rangle\right|=\left|\phi(\sum_{i=1}^n x_i\underline{\otimes}y_i)\right|=\left|\phi(\sum_{i=1}^n \lambda_ix_i\underline{\otimes} y_i)\right|=\left|\sum_{i=1}^n\lambda_i\left\langle Ux_i|y_i\right\rangle\right|
\]
\[
\leq \sum_{i=1}^n\lambda_i\left|\left\langle Ux_i|y_i\right\rangle\right|\leq \sum_{i=1}^n\lambda_i\left[\left\|x_i\underline{\otimes} y_i\right\|_1+\left\|A_2x_i\underline{\otimes} B_2y_i\right\|_1+...+\left\|A_nx_i\underline{\otimes}B_ny_i\right\|_1\right]
\]
\begin{equation}\label{lambda}
\leq \sum_{i=1}^n\lambda_i= 1.
\end{equation}
Let $E=span\left\{(x\otimes y,A_2x\otimes B_2y,\ldots,A_nx\otimes B_ny)\;|\;x,y\in H\right\}\subseteq H\otimes H\oplus H\otimes H\oplus\ldots\oplus H\otimes H$.

By \eqref{lambda}, we can find a bounded linear functional $\phi$ on $E$ with
\[
\left\|\phi\right\|=sup\left\{\left|\phi(t)\right|/ \left\|t\right\|\;|\;t\in E,t\neq 0\right\}
\]
\[
=sup\left\{\left|\phi(t)\right|\;|\;t\in E,\left\|t\right\|\leq 1\right\}
\]
\[
\leq 1,
\]
such that,
\[
\left\langle Ux|y\right\rangle=\phi((x\otimes y,A_2x\otimes B_2y,\ldots,A_nx\otimes B_ny)),\text{ }x,y\in H.
\]
Hence by the Hahn-Banach Theorem there is an extension $\widetilde{\phi}$ of $\phi$ to all $H\otimes H\oplus H\otimes H\oplus\ldots\oplus H\otimes H$ with 
$\left\|\widetilde{\phi}\right\|=\left\|\phi\right\|\leq 1$.
If we set,
\[
\left\langle U_1x|y\right\rangle=\widetilde{\phi}((x\otimes y,0,\cdots,0)),
\]
\[
\left\langle U_2x|y\right\rangle=\widetilde{\phi}(0,A_2(x)\otimes B_2(y),\cdots,0),
\]
\[
\vdots
\]
\[
\left\langle U_nx|y\right\rangle=\widetilde{\phi}(0,0,\cdots,A_n(x)\otimes B_n(y)),
\]
then by construction we have
\[
U=U_1+U_2+\cdots+U_n
\]
and
\[
\left|\left\langle U_1x|y\right\rangle\right| \leq \left\|\widetilde{\phi}\right\|\left\|x\right\|\left\|y\right\|\leq \left\|x\right\|\left\|y\right\|,
\]
\[
\left|\left\langle U_2x|y\right\rangle\right| \leq \left\|\widetilde{\phi}\right\|\left\|A_2x\right\|\left\|B_2y\right\|\leq \left\|A_2x\right\|\left\|B_2y\right\|,
\]
\[
\vdots
\]
\[
\left|\left\langle U_nx|y\right\rangle\right| \leq \left\|\widetilde{\phi}\right\|\left\|A_nx\right\|\left\|B_ny\right\|\leq \left\|A_nx\right\|\left\|B_ny\right\|.
\]
\item $(a)\Rightarrow(b)$: Assume $(a)$. If $(b)$ does not hold, we can choose $T_0\in \Delta\setminus conv(K)$. Since $conv(K)$ is closed, there exists by the Hahn-Banach Theorem a functional $\phi$ on $T(H)=B(H)^*$, such that 
\[
sup\left\{Re\phi(T)|\,T\in conv(K)\right\} < Re\;\phi(T_0).
\]
Since $T\in K\Rightarrow\gamma T\in K$ for all $\gamma \in \C$ with $\left|\gamma\right|=1$, we have
\[
sup\left\{(Re\phi(T))|{\;T\in conv(K)}\right\}=sup\left\{\left|\phi(T)\right||{\;T\in conv(K)}\right\} \geq 0.
\]
Moreover, 
\[
Re\;\phi(T_0)\leq\left|\phi(T_0)\right|.
\]
Hence,
\[
sup\left\{\left|\phi(T)\right||{\;T\in conv(K)}\right\}<\left|\phi(T_0)\right|.
\]
By replacing $\phi$ by a positive multiple of $\phi$ we can without loss of generality, assume that 
\begin{equation}\label{convex}
sup\left\{\left|\phi(T)\right||{\;T\in conv(K)}\right\}\leq 1 <\left|\phi(T_0)\right|.
\end{equation}
Using the standard duality $T(H)^*=B(H)$ there is a unique $U\in B(H)$, such that 
\[
\phi(T)=Tr(UT), \forall T\in T(H).
\]
By \eqref{convex}, we have for $x,y \in H$ satisfying 
\begin{equation}\label{crozy}
\left\|x\right\|\left\|y\right\|+\left\|A_2x\right\|\left\|B_2y\right\|+\ldots+\left\|A_nx\right\|\left\|B_ny\right\|=1,
\end{equation}
that $T=x\underline{\otimes}\bar{y}\in K$ and thus 
\begin{equation}\label{thus}
\left|\left\langle Ux|y\right\rangle\right|=\left|Tr(U(x\underline{\otimes}\bar{y}))\right|=\left|\phi(x\otimes \bar{y})\right|\leq 1.
\end{equation}
and since $\eqref{crozy}\Rightarrow\eqref{thus}$ we have by linearity in $x$, that 
\[
\left|\left\langle Ux|y\right\rangle\right|\leq \left\|x\right\|\left\|y\right\|+\left\|A_2x\right\|\left\|B_2x\right\|+\ldots+\left\|A_ny\right\|\left\|B_ny\right\|,\forall x,y\in H.
\]
However,
\begin{equation}\label{however}
\left|Tr(UT_0)\right|=\left|\phi(T_0)\right|> 1.			
\end{equation}
Now by (a), $U$ has a decomposition:
\[
U=U_1+U_2+\cdots+U_n.
\]
Therefore
\begin{equation}\label{final}
\left|Tr(UT_0)\right|\leq \left|Tr(U_1T_0)\right|+\left|Tr(U_2T_0)\right|+\ldots+\left|Tr(U_nT_0)\right|
\end{equation}

also from (a):
\[
\left|\left\langle U_1x|y\right\rangle \right| \leq \left\|x\right\|\left\|y\right\|
\]
for all $x,y$ in $H$.

Hence
\[
\left\|U_1\right\|=sup\left\{\left|\left\langle U_1x|y\right\rangle\right|:\left\|x\right\|\leq 1,\left\|y\right\|\leq 1\right\}\leq 1.
\]
Now by using Theorem 1.51(e), we get
\begin{equation}
\left|Tr(U_1T_0)\right|\leq \left\|U_1\right\|\left\|T_0\right\|_1\leq \left\|T_0\right\|_1
\end{equation}
where $U_1\in B(H)$ and $T_0\in L^1(H)$.

For the second term i.e. $\left|Tr(U_2T_0)\right|$ in \eqref{final}, if we define a new sesquilinear form $v$ as
\[
v(x,y):=u_2(A^{-1}_{2}x,B^{-1}_{2}y);
\]
then from Lemma \ref{bijective} there is $V\in B(H)$ such that 
\[
v(x,y)=u_2(A^{-1}_{2}x,B^{-1}_{2}y)=\left\langle Vx|y\right\rangle
\]
for all $x,y \in H$.

Therefore also by Lemma \ref{bijective}, there is $U_2\in B(H)$ satisfying
\[
u_2(x,y)=\left\langle U_2x|y\right\rangle=v(A_2x,B_2y)=\left\langle VA_2x|B_2y\right\rangle
\]
for all $x,y \in H$.

Hence,
\begin{align*}
Tr(U_2T_0)&=Tr(\sum U_2x_j\underline{\otimes} y_j)\\
&=(\sum \left\langle U_2x_j|y_j\right\rangle)\\
&=\sum\left\langle VA_2x_j|B_2y_j\right\rangle\\
&=Tr\left(\sum VA_2x_j\underline{\otimes} B_2y_j\right)\\
&=Tr\left(V\left[\sum A_2x_j\underline{\otimes} B_2y_j\right]\right)\\
&= Tr (VA_2T_0B_2^*).
\end{align*}
Therefore,
\[
\left|Tr(U_2T_0)\right|=\left|Tr(VA_2T_0B_2^*)\right|\leq \left\|V\right\|\left\|A_2T_0B_2^*\right\|_1.
\]
From the definition of $v$, we can easily get that
\[
\left\|V\right\|=sup\left\{\left|\left\langle VA_2x|B_2y\right\rangle\right|:\left\|A_2x\right\|\leq 1,\left\|B_2y\right\|\leq 1\right\}
\]
\[
=sup\left\{\left|\left\langle U_2x|y\right\rangle\right|:\left\|A_2x\right\|\leq 1,\left\|B_2y\right\|\leq 1\right\}\leq 1
\]
(where we use $\left|\left\langle U_2x|y\right\rangle\right|\leq\left\|A_2x\right\|\left\|B_2y\right\|$, from (a)).

Thus,
\begin{equation}
\left|Tr(U_2T_0)\right|\leq\left\|A_2T_0B_2^*\right\|_1.
\end{equation}

Similarly for the rest of the terms (for $n\geq 3$) in the above inequality \eqref{final}. We can define a new sesquilinear form $w$ by,
\[
w(x,y):=u_n(A_n^-1x,B_n^-1y).
\]
Also, from Lemma \ref{bijective} there is $W\in B(H)$ such that,
\[
w(x,y)=u_n(A_n^-1x,B_n^-1y)=\left\langle Wx|y\right\rangle
\]
for all $x,y\in H$.

So, also by Lemma \ref{bijective} there is $U_n\in B(H)$ such that 

\[
\left\langle U_nx|y\right\rangle=u_n(x,y)=w(A_nx,B_ny)=\left\langle WA_nx|B_ny\right\rangle.
\]
Hence,
\begin{align*}
Tr(U_nT_0)&=Tr(\sum U_nx_j\underline{\otimes} y_j)\\
&=(\sum \left\langle U_nx_j|y_j\right\rangle)\\
&=\sum\left\langle WA_nx_j|B_ny_j\right\rangle\\
&=Tr\left(\sum WA_nx_j\underline{\otimes}B_ny_j\right)\\
&=Tr\left(W\left[\sum A_nx_j\underline{\otimes}B_ny_j\right]\right)\\
&= Tr (WA_nT_0B_n^*).
\end{align*}
Therefore,
\[
\left|Tr(U_nT_0)\right|=\left|Tr(WA_nT_0B_n^*)\right|\leq \left\|W\right\|\left\|A_nT_0B_n^*\right\|_1.
\]
Also from the definition of $w$, we can easily see that
\[
\left\|W\right\|=sup\left\{\left|\left\langle WA_nx|B_ny\right\rangle\right|:\left\|A_nx\right\|\leq 1,\left\|B_ny\right\|\leq 1\right\}
\]
\[
=sup\left\{\left|\left\langle U_nx|y\right\rangle\right|:\left\|A_nx\right\|\leq 1,\left\|B_ny\right\|\leq 1\right\}\leq 1,
\]
(where we use $\left|\left\langle U_nx|y\right\rangle\right| \leq \left\|A_nx\right\|\left\|B_ny\right\|$, from (a)).

Thus,
\begin{equation}
\left|Tr(U_nT_0)\right|\leq\left\|A_nT_0B_n^*\right\|_1.
\end{equation}
Finally from inequality \eqref{final},
\[
\left|Tr(UT_0)\right|\leq \left|Tr(U_1T_0)\right|+\left|Tr(U_2T_0)\right|+\ldots+\left|Tr(U_nT_0)\right|
\]
\[
\leq \left\|T_0\right\|_1+\left\|A_2T_0B_2^*\right\|_1+\ldots+\left\|A_nT_0B_n^*\right\|_1\leq 1.
\]
But this contradicts \eqref{however}. Thus $(a)$ does not hold and we have proved that $(a)\Rightarrow(b)$.
\end{proof}



\section{A counterexample to decomposing bilnear forms into three bilinear forms} \label{SectJ}

\noindent

In this section, we will use the criterion established in the previous section

  to give a counterexample which will show that the decomposition of a bilinear form into three bounded terms is not always possible. We will start with a useful lemma.

\begin{lemma}\label{ch6}
Let $H$ be a finite dimensional Hilbert space. If $S$ is in $T(H)_+=B(H)_+$ and 
\[
S=\sum_{j=1}^m x_j\underline{\otimes} y_j 
\]
is a finite rank operator on $H$, (see the appendix for the definition) 
such that
\begin{equation}\label{Fin}
\left\|S\right\|_1=\sum_{j=1}^m\left\|x_j\right\|\left\|y_j\right\| 
\end{equation}
then each $y_j$ is a positive multiple of $x_j$.
\end{lemma}
\subsection{A counterexample}
We use Theorem \ref{ch5} to build the counterexample. We prove that the condition
\[
conv(K)=\Delta
\]
with
\[
K=\left\{x\underline{\otimes}y:\;\left\|x\right\|\ \left\|y\right\|+\left\|Ax\right\|\left\|By\right\|+ \left\|Cx\right\|\left\|Dy\right\|\leq 1\right\}
\]
and
\[
\Delta=\left\{T\in B(H):\;\left\|T\right\|_1+\left\|ATB^*\right\|_1+\left\|CTD^*\right\|_1\leq 1\right\}
\]
is not always true. Therefore the decomposition into three terms fails. Consider the Hilbert space $H=\C^2$. Consider the operators $B=D=I$ and $A,C$ positive invertible and not commuting. In particular $A$ and $C$ don't have any common eigenvectors. For instance we could choose
\[
A=E^2=
\left[ {\begin{matrix}
\sqrt{2} & 0 \\ 
0 & 1
\end{matrix} } \right]
\left[\begin{matrix}
\sqrt{2} & 0 \\ 
0 & 1
\end{matrix}\right]
=
\left[ {\begin{matrix}
2 & 0 \\
0 & 1
\end{matrix} }\right]
\]
and
\[
C=F^*F=
\left[ {\begin{matrix}
1 & 1 \\ 
0 & 1
\end{matrix} } \right]
\left[\begin{matrix}
1 & 0 \\ 
1 & 1
\end{matrix}\right]
=
\left[ {\begin{matrix}
2 & 1 \\
1 & 1
\end{matrix} }\right]
\]
Put $c:=\,(\|1\|_1 + \|A\|_1 + \|C\|_1)^{-1}$ and take $T=c1$.
Now we will show that $T\in \Delta$ but $T\notin \mbox{conv}K$. Let $U\in B(H)$ such that 
\[
U=
\left[\begin{matrix}
5 & 1 \\ 
1 & 3
\end{matrix}\right].
\]
Consider $u$ the corresponding sesquilinear form, then
\[
u(x,y):=\left\langle Ux|y\right\rangle.
\]
Now,
\[
Ax=\left[ {\begin{matrix}
2x_1\\
x_2
\end{matrix} }\right],\;
By=Iy=\left[ {\begin{matrix}
y_1\\
y_2
\end{matrix} }\right]
\]
and
\[
Cx=
\left[ {\begin{matrix}
2x_1+x_2\\
x_1+x_2
\end{matrix} }\right],\:
Dy=Iy=
\left[ {\begin{matrix}
y_1\\
y_2
\end{matrix} }\right].
\]
Then,
\[
\left|u(x,y)\right|=\left|5x_1y_1+x_2y_1+x_1y_2+3x_2y_2\right|
\]
\[
=\left|x_1y_1+x_2y_2+2x_1y_1+x_2y_2+(2x_1+x_2)y_1+(x_1+x_2)y_2\right|
\]
\[
\leq \left\|x\right\|\left\|y\right\|+\left\|Ax\right\|\left\|By\right\|+\left\|Cx\right\|\left\|Dy\right\|.
\]
From the definition of $\Delta$,
\[
\Delta=\left\{T\in B(H)|\left\|T\right\|_1+\left\|ATB^*\right\|_1+\left\|CTD^*\right\|_1\leq 1\right\}.
\]
For $T=c1$, we find
\begin{align*}
\left\|T\right\|_1+\left\|ATB^*\right\|_1+\left\|CTD^*\right\|_1&=c\left\|1\right\|_1+c\left\|A.1\right\|_1+c\left\|C.1\right\|_1\\
&=c\left(\left\|1\right\|_1+\left\|A\right\|_1+\left\|C\right\|_1\right)\\
&=(\|1\|_1 + \|A\|_1 + \|C\|_1)^{-1}\left(\left\|1\right\|_1+\left\|A\right\|_1+\left\|C\right\|_1\right)= 1
\end{align*}
Therefore,
\[
T\in \Delta.
\]
It is not difficult to see that the operators $T$, $AT$ and $CT$ are positive. In fact
\[
T=\left|T\right|=cI
\]
and 
\[
AT=A(cI)=cA\quad\text{and}\quad CT=C(cI)=cC
\]
Also we have,
\begin{align*}
N_1&:=\left\|T\right\|_1=c\left\|I\right\|_1=2c=1/4,\\
N_2&:=\left\|AT\right\|_1=c\left\|AI\right\|_1=c\left\|A\right\|_1=3/8,\\
and\;
N_3&:=\left\|CT\right\|_1=c\left\|CI\right\|_1=c\left\|C\right\|_1=3/8
\end{align*}
are of course non-zero positive numbers. Now suppose that,

\[
T=\sum_{j=1}^n\lambda_jx_j\underline{\otimes} y_j\quad \in conv(K),
\]
i.e. 
\[
\sum_{j=1}^n\lambda_j=1\quad\text{and}\quad x_j\underline{\otimes} y_j\in K\quad\text{for all}\quad1\leq j \leq n.
\]
We know $T\in \Delta$, in fact 
\[
\left\|T\right\|_1+\left\|AT\right\|_1+\left\|CT\right\|_1=1.
\]
Moreover,

\begin{equation}\label{eq1}
\left\|T\right\|_1\leq\sum_{j=1}^n\lambda_j\left\|x_j\right\|\left\|y_j\right\|:=M_1,
\end{equation}
\begin{equation}\label{eq2}
\left\|AT\right\|_1\leq\sum_{j=1}^n\lambda_j\left\|Ax_j\right\|\left\|y_j\right\|:=M_2,
\end{equation}
\begin{equation}\label{eq3}
\left\|CT\right\|_1\leq\sum_{j=1}^n\lambda_j\left\|Cx_j\right\|\left\|y_j\right\|:=M_3.
\end{equation}
Also,
\[
x_j\underline{\otimes}y_j\in K,
\]
whence
\[
\left\|x_j\right\|\left\|y_j\right\|+\left\|Ax_j\right\|\left\|y_j\right\|+\left\|Cx_j\right\|\left\|y_j\right\|\leq 1.
\]
Therefore,
\[
\lambda_j\left\|x_j\right\|\left\|y_j\right\|+\lambda_j\left\|Ax_j\right\|\left\|y_j\right\|+\lambda_j\left\|Cx_j\right\|\left\|y_j\right\|\leq\lambda_j
\]
\[
\Longrightarrow
\]
\[ \sum_{j=1}^n\lambda_j\left\|x_j\right\|\left\|y_j\right\|+\sum_{j=1}^n\lambda_j\left\|Ax_j\right\|\left\|y_j\right\|+\sum_{j=1}^n\lambda_j\left\|Cx_j\right\|\left\|y_j\right\|\leq\sum_{j=1}^n\lambda_j=1.
\]
Hence
\[
\sum_{j=1}^n\lambda_j\left\|x_j\right\|\left\|y_j\right\|+\sum_{j=1}^n\lambda_j\left\|Ax_j\right\|\left\|y_j\right\|+\sum_{j=1}^n\lambda_j\left\|Cx_j\right\|\left\|y_j\right\|=1.
\]
All the above inequalities \eqref{eq1}, \eqref{eq2} and \eqref{eq3} are equalities since the system,
\[
(M_1-N_1)+(M_2-N_2)+(M_3-N_3)=0,
\]
and
\[
N_1\leq M_1
\]
\[
N_2\leq M_2
\]
\[
N_3\leq M_3
\]
has only the trivial solution,
\[
M_1=N_1
\]
\[
M_2=N_2
\]
\[
M_3=N_3
\]
i.e.
\[
\left\|T\right\|_1=\sum_{j=1}^n\lambda_j\left\|x_j\right\|\left\|y_j\right\|,
\]
\[
\left\|AT\right\|_1=\sum_{j=1}^n\lambda_j\left\|Ax_j\right\|\left\|y_j\right\|,
\]
\[
\left\|CT\right\|_1=\sum_{j=1}^n\lambda_j\left\|Cx_j\right\|\left\|y_j\right\|.
\]
Applying Lemma \ref{ch6} to the positive operator $T$, we find

\begin{equation}\label{alpha}
y_j=\alpha_j(\lambda_jx_j)
\end{equation}
where $\alpha_j$ is a positive scalar. We can also apply Lemma \ref{ch6} to the positive operators $AT$ and $CT$ to get

\begin{equation}\label{beta}
y_j=\beta_j(\lambda_jAx_j)
\end{equation}
where $\beta_j$ is also a positive scalar. and

\begin{equation}\label{gamma}
y_j=\gamma_j(\lambda_jBx_j)
\end{equation}
for another positive scalar $\gamma_j$.

Now from \eqref{alpha}, \eqref{beta} and \eqref{gamma}, we have
\[
y_j=\lambda_j\alpha_jx_j=\lambda_j\beta_jAx_j=\lambda_j\gamma_jCx_j.
\]
Hence
\[
Ax_j=(\alpha_j/\beta_j)x_j
\]
and
\[
Cx_j=(\alpha_j/\gamma_j)x_j.
\]

Therefore $x_j$ is a common eigenvector for operators $A$ and $C$ but this contradicts our assumption on $A,C$.

Thus
\[
T\notin conv(K).
\]

\section*{Appendix}
This appendix supposed to give a brief overview on tensor products and finite rank operators (cf. \cite{C00}, \cite{R02} and \cite{D93}).\\

\noindent
\zw{1}{Rank one operator}
\label{rank} If $g,h$ are elements of a Hilbert space $H$ we define the operator $g\underline{\otimes} h$ on $H$ by 
\[
(g\underline{\otimes} h)(f)= \left\langle f|h\right\rangle g
\]
Clearly, $\left\|g\underline{\otimes} h\right\|=\left\|g\right\|\left\|h\right\|$.
$g\underline{\otimes} h$ is rank one operator if $g$ and $h$ are non-zero.
If $g$,$\acute{g},h,\acute{h}\in H$ and $T\in B(H)$, then the following equalities are readily verified:
\begin{align*}
(g\underline{\otimes} \acute{g})(h\underline{\otimes}\acute{h})&=\left\langle h|\acute{g}\right\rangle(g\underline{\otimes}\acute{h})\\
(g\underline{\otimes} h)^*&=h\underline{\otimes}g\\
T(g\underline{\otimes}h)&=Tg\underline{\otimes}h\\
(g\underline{\otimes} h)T&=g\underline{\otimes}T^*h
\end{align*}
The operator $g\underline{\otimes}g$ is a rank-one projection if and only if $\left\langle g,g\right\rangle=1$, that is, $g$ is a unit vector. Conversely, every rank-one projection is of the form $g\underline{\otimes}g$ for some unit vector $g$. Indeed, if $e_1,.....,e_n$ is an orthonormal set in $H$, then the operator $\sum^n_{j=1}e_j\underline{\otimes}e_j$ is the orthogonal projection of $H$ onto the vector subspace $\C.e_1+.....+\C.e_n$.
If $T\in B(H)$ is rank-one operator and $g$ a non-zero element of its range, then $T=g\underline{\otimes}h$ for some $h\in H$. For if $f\in H$, then $T(f)=\tau{(f)}g$ for some scalar $\tau(f)\in \C$. It is readily verified that the map $f\rightarrow \tau(f)$ is a bounded linear functional on $H$, and therefore, by the Riesz Representation Thorem, there exists $h\in H$ such that $\tau(f)=\left\langle f|h\right\rangle$ for all $f \in H$. Therefore, $T=g\underline{\otimes}h$.\\

\zw{2}{Finite rank operator}
\label{Finite} Let $H$ be Hilbert space. An operator T on $H$ is said to be of finite rank (of rank m) if $R(T)$ is finite-dimensional (m-dimensional).\\

\zw{3}{The Projective Norm}
\label{projective} Let $X$ and $Y$ be Banach spaces. How should we put a norm on the tensor product $X\otimes Y$? Consider first the elementary tensors. It is natural to require that
\[
\left\|x\otimes y\right\|\leq \left\|x\right\|\left\|y\right\|.
\]
Now let $u$ be any element of $X\otimes Y$. If $\sum_{i=1}^n x_i \otimes y_i$ is a representation of $u$, then it follows from the Triangle Inequality that the norm must satisfy 
\[
\left\|u\right\|\leq\sum_{i=1}^n\left\|x_i\right\|\left\|y_i\right\|.
\]
Since this holds for every representation of $u$, it follows that 
\[
\left\|u\right\|\leq\inf\left\{\sum_{i=1}^n\left\|x_i\right\|\left\|y_i\right\|\right\},
\]
the infimum being taken over all representations of $u$. The right hand side of this inequality is the biggest possible candidate for a ``natural" norm on $X\otimes Y$. This norm is known as the projective norm and is defined as follows:
\[
\pi(u)=inf\left\{\sum_{i=1}^{n}\left\|x_i\right\|\left\|y_i\right\|\colon u=\sum^{n}_{i=1}{x_i\otimes y_i}\right\}.
\]
If it is necessary to specify the component spaces in the tensor product, we shall denote this norm by $\pi_{X,Y}(u)$, or $\pi(u;X\otimes Y)$.\\

\zw{4}{Corollary}

For $u\in X\hat{\otimes_\pi}Y$
\[
\pi(u;X\otimes Y)=\inf\left\{\sum^\infty_{n=1}\left\|x_n\right\|\left\|y_n\right\|:\;u=\sum^\infty_{n=1}x_n\otimes y_n\right\}.
\].\\

\zw{5}{Nuclear operator}

 An operator $T:X\rightarrow Y$ between Banach spaces is called nuclear if there are $x'_i\in X'$, where $X'$ is a dual space of $X$ and $y_i\in Y$ such that $\sum^\infty_{i=1} \left\|x'_i\right\|\left\|y_i\right\|<\infty$ and for each $x\in X$,
\[
Tx=\sum_{i=1}^{\infty}\left\langle x'_i|x\right\rangle y_i.
\]
Clearly, $T$ is continuous in this case and 
\[
N(T):=\inf\left\{\sum^\infty_{i=1}\left\|x'_i\right\|\left\|y_i\right\|:\;T(\cdot)=\sum_{i=1}^{\infty}\left\langle x'_i|\cdot\right\rangle y_i.
\right\}
\]
defines a norm on the vector space $\aleph(X,Y)$ of all nuclear operators; it easy to see that $\aleph(X,Y)$ becomes a Banach space with this norm.\\

\zw{6}{The map}
The definition of $\aleph(X,Y)$, its norm and Corollary A.5 give that the map in the following theorem is surjective.
\[
J:\acute{X}\otimes_\pi Y\rightarrow \aleph(X,Y) 
\]
\[
x'\otimes y\rightarrow x'\underline{\otimes}y
\].\\

\zw{7}{remark}
If $X=Y=H$ be a Hilbert space, then $\aleph(X,Y)=B_1(H)$.\\

\zw{8}{remark}
The map in remark A.6 in case $X=Y=H$ be a Hilbert space gives an isometric isomorphism.


\section*{Acknowledgements}

\noindent
I am indebted to many people for their long-lasting support and encouragement which was invaluable for the successful completion of this research work. I must thank Prof. Joachim Cuntz who initiated the article and granted me the EU-scholarship under the project ``ToDyRic''. The discussions with him were very helpful and I appreciate his ideas and suggestions a lot. I am deeply grateful to Prof. Uffe Haagerup for his valuable hints. I would also like to thank the research group in Functional Analysis and Operator Algebras in M\"unster, in particular Siegfried Echterhoff, Christian Voigt and Aaron Tikuisis.

\end{document}